\documentclass{article}

\usepackage{url}
\usepackage{tikz}
\usetikzlibrary{shapes,arrows}
\tikzstyle{block} = [draw, fill=blue!10, rectangle, 
    minimum height=3em, minimum width=6em]
\tikzstyle{sum} = [draw, fill=blue!20, circle, node distance=1cm]
\tikzstyle{input} = [coordinate]
\tikzstyle{output} = [coordinate]
\tikzstyle{pinstyle} = [pin edge={to-,thin,black}]

\usepackage{amsmath} %
\usepackage{amsfonts, latexsym, amssymb,amsthm} 
\usepackage{color}
\usepackage{graphicx}

\usepackage{hyperref}
\usepackage{pdflscape}
\usepackage{subcaption}
\usepackage{placeins}
\usepackage{etex,etoolbox}
\usepackage{thmtools}
\usepackage{environ}

\DeclareGraphicsExtensions{.pdf,.png,.gif,.jpg}

\theoremstyle{plain}
\newtheorem{theorem}{Theorem}
\newtheorem{assumption}{Assumption}
\newtheorem{corollary}{Corollary}
\newtheorem{lemma}{Lemma}

\theoremstyle{definition}

\newtheorem{example}{Example}

\theoremstyle{remark}
\newtheorem{remark}{Remark}

\newcommand{\iid}{i.\@i.\@d.\@ }
\newcommand{\eg}{\emph{e.g.}, }
\newcommand{\ie}{\emph{i.e.}, }
\newcommand{\ud}{\mathrm{d}}

\providecommand{\abs}[1]{\left\lvert#1\right\rvert}

\providecommand{\eps}{\varepsilon}
\newcommand{\E}{{\mathbb E}}

\renewcommand{\P}{{\mathbb P}}

\newcommand{\neww}[1]{}

\renewcommand{\hat}{\widehat}
\renewcommand{\leq}{\leqslant}
\renewcommand{\geq}{\geqslant}

\begin{document}

 \title{Signalling and obfuscation for congestion control}

  \author{
 Jakub Mare{\v c}ek,
     Robert Shorten,
     Jia Yuan Yu\thanks{Email: \{ jakub.marecek, robshort, jiayuanyu \}@ie.ibm.com}  \\ 
 IBM Research -- Ireland\\ Technology Campus Damastown, Dublin 15,
     Ireland 
}

\maketitle

\begin{abstract}
  We aim to reduce the social cost of congestion in many smart city
  applications.  In our model of congestion, agents interact over
  limited resources after receiving signals from a central agent that
  observes the state of congestion in real time.  Under natural models
  of agent populations, we develop new signalling schemes and show that
  by introducing a non-trivial amount of uncertainty in the signals,
  we reduce the social cost of congestion, i.e., improve social
  welfare.  The signalling schemes are efficient in terms of both
  communication and computation, and are consistent with past
  observations of the congestion.  Moreover, the resulting population
  dynamics converge under reasonable assumptions.
\end{abstract}

\section{Introduction}
The study of ``Smarter Cities'' and ``Smarter Planet'' provides a host
of new and challenging control engineering problems. Many of these
problems can be cast in a congestion control framework, where a
large number of agents such as people, cars, or consumers
 compete for a limited resource.
Examples of such problems include consumers competing
for electricity supply; road users competing for space in the roads;
users of bike (cars) sharing systems competing for access to bikes (cars);
and pedestrians competing for access. 
Within each such problem, there are many variants. Pedestrians could be,
for example, workers arriving at work through a common gate, 
people leaving a stadium, waiting to check in at an airport,
or in an emergency. 
Each variant may have additional constraints, but certain key features remain the same.

Addressing these problems is hugely challenging, even without considering the issues of scale and limited communication.
Resource utilisation should be minimised while delivering a certain quality of service to individual agents.
Congestion is often caused by bursty arrivals, rather than the capacity and quality-of-service constraints \emph{per se}.
Prediction systems to alleviate congestion (informing customers of parking spaces, for example) create
 complicated feedback systems that are difficult to model and control.
This latter issue of prediction and optimisation under feedback
 clearly opens a wide area of research,
with links to reinforcement learning, adversarial game theory,
and closed-loop identification at scale.

An example of congestion due to synchronised demand is the well-known
flapping effect in road networks, parking lots, and bike-sharing
stations.
When presented with alternative choices of resources,
agents who choose greedily based on past observations cause a
congestion to oscillate between the resources over time.
The flapping effect occurs due to actuation delays and the feedback effect between
agents' observations and actions: the agents' choices are based on
\emph{past} observations, but affect the \emph{future} state of congestion.
Our principal motivation is to develop methods to break up this effect
by taking into account these feedback issues and actuation delays. The starting point for our work
is a simple model of agent-induced congestion.

Our objective in this paper is to take a first step toward addressing some of these problems. 
We wish to develop easy-to-implement algorithms, i.e. not requiring much inter-agent communication or dedicated infrastructure, 
 with the objective of ``de-synchronising'' agents' actions by providing them with signals, whereby achieving a temporal load balancing over networks.
Load-balancing ideas in this direction have been recently suggested in
a variety of applications
\cite{ScholteKing2014,ScholteChen2014,6545321}.
However, the work presented here goes far beyond what has been proposed in several ways.
First, it suggests non-trivial signalling schemes, as opposed to the simple randomisation and differential pricing ideas.
Second, it takes into account heterogenous agent behaviour and actions.
Furthermore, our  algorithms are provably scalable and can be analysed in a simple stochastic setting, 
whereby yielding provable behaviour even in the case of large communication and actuation delays.

Specifically, we model a congestion problem as a multiagent system evolving over time, where
the agents follow natural policies. 
Roughly speaking, we show that by varying the amount of uncertainty agents face, one can
reduce congestion in a controlled manner, and that a desirable state is arrived at asymptotically.
Our investigation proceeds as follows. First, we fix the congestion
cost functions, as well as the policies of each agent in a
heterogeneous population.  Then, we vary the parameters controlling
the distribution of the random signals sent to the agents and evaluate
the resulting social cost. We repeat this investigation for two
signalling schemes, and show that the social cost is optimised
by introducing uncertainty.  Although
our results are not difficult to derive, they are -- to the best of our
knowledge -- both novel and
useful in a smart city context.

The paper is organised as follows.  After describing the
model of congestion in Section~\ref{sec:mod1}, we
set our work in
context with the existing works in Section~\ref{sec:rel}.
Next, we present in two sections our main results for two models of signalling
and agent-response.
Section~\ref{sec:mod2} considers a scalar signalling scheme and agents with 
different actuation delays. 
Section~\ref{sec:mod3} considers a broadcast interval signalling scheme and agents with
different risk aversion. 
In each of these two sections, we show through theoretical and
empirical analysis that withholding a non-trivial
amount of information from the agents
reduce the overall social cost.
Finally, we present open
questions in Section~\ref{sec:con}.

\section{Model}
\label{sec:mod1}

We consider a model of congestion, where a population of $N$
agents $\{1,\ldots,N\}$ is confronted with two alternative choices at discrete time
steps. Note this situation is widespread and is generalised to the case if multiple choices (more than two)
in a straightforward manner\protect\footnote{Our results can be
generalised to the case of arbitrarily many choices. The main change
consists of
replacing binomial probability distributions by multi-nomial ones.}. Our approach is
to model congestion using probabilistic techniques. In this context
alternative agent actions are denoted by $\{A,B\}$, and time
steps are denoted by $t=1,2,\ldots$.  Then, the random variable
$a_t^i$ denotes the choice of agent $i$ at time $t$ and $n^A_t = \sum_i
1_{[a_t^i = A]}$ be the number of agents choosing action $A$ at time
$t$.  Throughout the paper, we assume that each agent has to pick an
action at every time $t$, i.e., the number of agents choosing action $B$
is $n^B_t = N-n^A_t$.

\subsection{Costs} 
We also assume that each action has a cost. For example, this could be total trip time or fuel consume.
The total cost of action $A$ at time $t$ is a function of the
number $n^A_t$ of agents that pick $A$ at time $t$.
Let
$c_A: \mathbb R_+ \to \mathbb R_+$ denote the so-called cost function
for action $A$.  If $n^A_t$ agents choose action $A$ at time $t$, the
cost to each of these agents is $c_A(n^A_t)$.
We define $c_B$ similarly to $c_A$.  Figure~\ref{fig:ushaped5settings_funs} gives an example of these
cost functions.
\begin{figure}[t!]
  \begin{subfigure}{0.5\textwidth}
    \includegraphics[clip=true,width=0.99
    \textwidth]{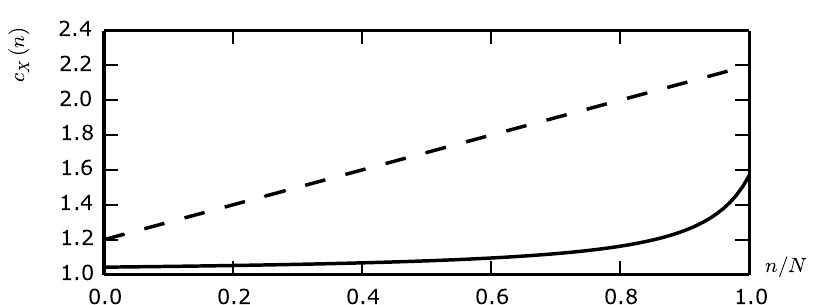}
    \caption{Cost functions: $c_A(n) \triangleq 1.2 + n/N$ (dashed line)
      and $c_B(n) \triangleq 1 + (1.08-n/N)^{-1}/22$ (solid line).\\ \vspace{0.87em}}
    \label{fig:ushaped5settings_funs}
  \end{subfigure}
  \begin{subfigure}{0.5\textwidth}
    \includegraphics[clip=true,width=0.99
    \textwidth]{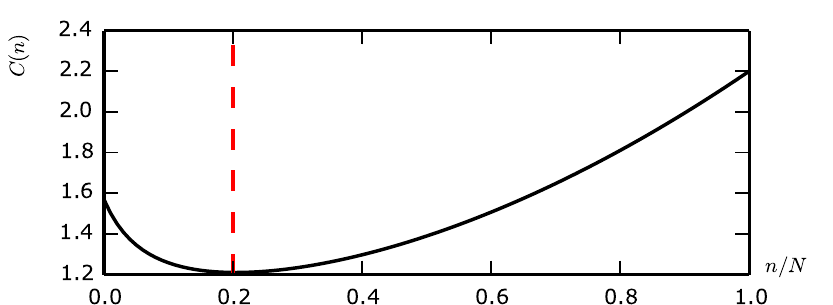}
    \caption{One-step social cost as a function of the fraction of
      agents choosing action $A$.  The dashed vertical line highlights the
      minimum at a population profile of $n^*/N = 0.2$.}
    \label{fig:ushaped5settings_costs}
  \end{subfigure}
  \caption{The setting of the running example. Despite the
    fact that the curves do not intersect, the optimum of the social cost
    is achieved when both actions are used -- with a $0.2$-fraction of
    agents choosing action $A$.}
  \label{fig:ushaped5settings}
\end{figure}
The {\em social cost}  $C(n^A_t)$ scales the costs of the two actions at
time $t$ with the proportions of agents taking the two actions, i.e.,
\begin{align}
  \label{eqn:socialCost}
  C(n^A_t) \triangleq \frac{n^A_t}{N} \cdot c_A(n^A_t) +
  \frac{n^B_t}{N} \cdot c_B(n^B_t).
\end{align}
The social cost corresponding to the cost functions of
Figure~\ref{fig:ushaped5settings_funs} is shown in
Figure~\ref{fig:ushaped5settings_costs}. We also define the
{\em time-averaged social cost} as follows:
\begin{align}
  \label{eqn:meanSocialCost}
  \hat C_T \triangleq \frac{1}{T} \sum_{t=1}^{T} C(n^A_t),
\end{align}
which will be used to illustrate the evolution of the social cost in
simulations.
The {\em social optimum} $n^*$ of the social cost $C$ is defined in the usual way:
\begin{align}
  n^* = \min_{n \in \{0,1, \ldots, N\}} \frac{n}{N} \cdot c_A(n) + 
  \frac{N - n}{N} \cdot c_B(N - n).
  \label{eq:socialopt}
\end{align}
Given this basic setting, we wish to control these various costs
by influencing individual agents' decisions. This is done using
simple signalling schemes. We consider two specific schemes.
\begin{itemize}
\item The first one de-synchronises 
greedy agents with different actuation delays, i.e., delays between 
receiving a signal and taking a corresponding action. 
\item The second one 
de-synchronises agents with various levels of risk aversion.
\end{itemize}
We will later show that under reasonable assumptions on the behaviour 
of agents, these signalling schemes lead to desirable outcomes.

\section{Related Work}\label{sec:rel}


Before proceeding, we now briefly mention some related work. Note that
the related work spans many fields 
and an exhaustive survey is neither possible nor intended here.

Congestion problems are important 
in networking, control,
  operations research, and game theory 
  \cite{sheffi1985urban,ibaraki1988resource,stefanov2001separable,patriksson2008survey}. Our
  work proposes signalling solutions that reduce congestion by introducing a non-trivial amount of
  uncertainty through randomisation.
Although it may seem surprising given the simplicity of the model,
our results are unknown in the literature to the best of our
knowledge. 

In operations research and mathematical optimisation, congestion problems
are often considered to be deterministic nonlinear resource
allocation problems (\eg 
\cite{sheffi1985urban,ibaraki1988resource,stefanov2001separable,patriksson2008survey}).
In our work, however, the decision-maker only chooses the signals to send, instead of
an allocating the resources, \emph{per se}.

Game theoretical models assume agents that act
strategically, i.e., each agent has a non-negligible effect on the
outcome.  The study of
one-shot congestion games has a long history, starting with \cite{Rosenthal1973}.
%
%
The socially optimal and Nash equilibrium outcomes of congestion games
have been compared in
\cite{RouTar02,price-of-stability}, and \cite{price-of-uncertainty}. 
These
  works show that when agents have full information, a resulting
  equilibrium outcome can incur much higher total congestion than a
  socially optimal outcome.
An important distinction of our work is that, instead of a
  static game theoretical model, we consider a dynamic model evolving
  over time. The agents do not observe other agents' actions, nor do
  they act strategically: the only information they have about the
  congestion of the resources is obtained from signals broadcast by a
  central agent.  
We model a heterogeneous population of agents, i.e.,
 a large number of agents with a variety of 
\emph{fixed} policies. 
These agents do not act strategically in the fashion of price-taking 
participants in large markets, which reflects situations where each agent has very limited 
impact on the outcome as a whole.
We take the view of the central
  agent, who has perfect knowledge of the congestion across resources
  up to the previous time step, but imperfect knowledge of the
  stochastic composition of the population of users. Our goal is to
  minimise the social welfare or total congestion: we study 
how the amount of information withheld can contribute to better social 
outcomes.
In contrast to evolutionary game theory, where agents of different types interact and 
the population profile (or distribution) evolves over time, the
population profile remains fixed in our setting, and only the action 
profile evolve over time.

Our random signals are reminiscent of perturbation
schemes in repeated games, such as Follow-the-Perturbed-Leader
\cite{Hannan57}, trembling-hand equilibrium \cite{Selten1975},
stochastic fictitious play \cite{Harsanyi1973},
or the power of two choices \cite{963420}.
When the agents use such randomised algorithms in their decision-making, the 
resulting demand process 
has been shown to be behave rather well in theory (e.g., \cite{963420}), as well as in 
a number of applications,
e.g., 
in parking \cite{ScholteKing2014},
bike sharing \cite{ScholteChen2014},
and 
charging electric vehicles \cite{6545321}.
However, we propose a signalling and guidance scheme that combines
randomization and intervals.

In the transportation literature, \cite{smith1979} introduces the notion
of equilibrium as the limit of the congestion distribution if it
exists.  \cite{horowitz1984} considers a number of notions of noisy
signals and studies greedy policies and equilibria.  Our approach is
also related to signalling of parking space availability \cite{SKCS13}.


Our interval signalling scheme is reminiscent of the equilibrium
outcome of \cite{CraSob82} in the context of signalling games in
economics (cf. \cite{Sobel2013} for an up-to-date survey).  However,
we consider the problem of optimal signalling in a dynamic system,
whereas \cite{CraSob82} considers Nash equilibrium signalling in
single-shot games.
For signalling games, \cite{Green2007} shows that more information does
not generally improve the equilibrium welfare of agents.  Their
notion of ``information,'' due to \cite{Blackwell51}, is however very
different from ours.

Finally, let us note that there is also a related draft \cite{MarecekInterval} by the present
authors, which explores the notion of $r$-extreme interval signalling 
and optimisation over truthful interval signals.





\section{Scalar Signalling}
\label{sec:mod2}

In this section, we model agents with different actuation delays -- \ie delays between 
receiving a signal and taking a corresponding action -- and
introduce a signalling scheme with the aim of desynchronising these agents.

\begin{figure}
\centering
\includegraphics[width=0.99\textwidth,clip=true,trim=8.5cm 8cm 5.5cm 4.5cm]{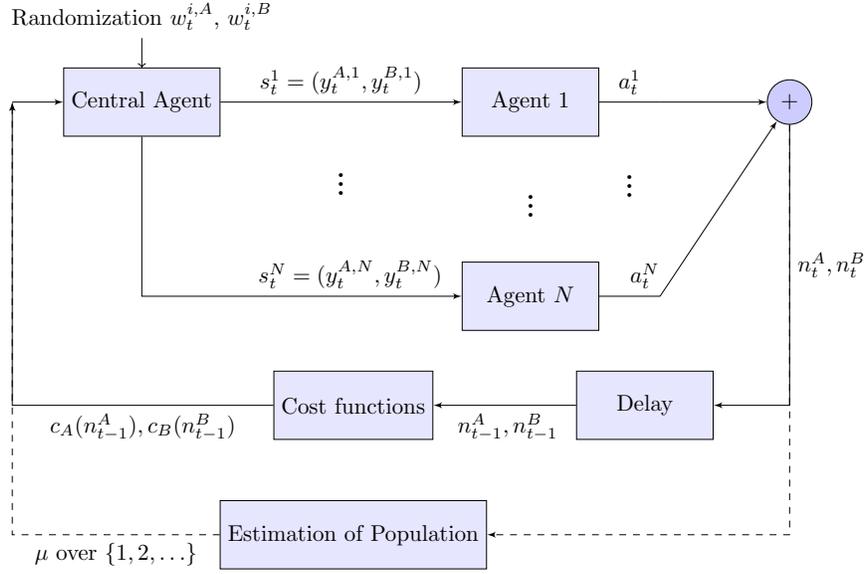}
  \caption{Block diagram for scalar signalling.}
  \label{fig:diag1}
\end{figure}

Notationwise, let $H_t$ denote the congestion costs at time $t$:
\begin{align*}
  H_t \triangleq (c_A(n^A_t) ,
  c_B(n^B_t)) \in \mathbb R^2.
\end{align*}
We write $s^i_t \triangleq s^i_t(H_{t-1})$ denotes the signal that the central agent sends to 
agent $i$ at every time $t$.
For a fixed integer $d$, a \emph{signalling scheme} is
a sequence of mappings
\begin{align}\label{eq:81}
  \{s^i_t : \mathbb R^2 \to \mathbb R^d \mid i = 1,\ldots,N,
  t=1,2,\ldots \}.
\end{align}

\subsection{$\sigma$-Scalar Signalling} 

We propose the following scalar signalling scheme, which corresponds
  to the case $d=2$ in \eqref{eq:81}, with a
parameter $\sigma > 0$. 
At every time step $t$, the central agent sends to each agent $i$ a
distinct signal $s^i_t := (y^{A,i}_t, y^{B,i}_t) \in \mathbb R^2$:
\begin{align}
  y^{A,i}_t &\triangleq c_A(n^A_{t-1}) + w^{i,A}_t, \\
  y^{B,i}_t &\triangleq c_B(n^B_{t-1}) + w^{i,B}_t,
\end{align}
where $\{(w^{i,A}_t,w^{i,B}_t) : i=1,\ldots,N; \;
t=1,2,\ldots\}$ are independent and identically distributed (\iid) zero-mean normal random variables such that
$\E (w^{i,A}_t - w^{i,B}_t)^2 = \sigma^2$ for all $i$ and $t$. 

\begin{remark}[Random Signals]
  The signals $\{s^i_t\}$ in this section and the next are generated by
  the random variables -- $\{w^{i,A}_t,w^{i,B}_t\}$ in this section and $\{\nu_t,\eta_t\}$ in
  the next.  In the rest of the paper,
  all expectations and probabilities are with respect to the
  distribution of these random signals.
\end{remark}

In contrast to signalling schemes that solve
an allocation problem to minimise the social cost and coerce each
agent towards its assigned
action, $\sigma$-scalar signalling 
is \emph{scalar-truthful} in the following sense:
\begin{align*}
  \E y^{A,i}_t &= c_A(n^A_{t-1}), \quad\mbox{and}\quad \E y^{B,i}_t = c_A(n^B_{t-1}),
\end{align*}
for all $i$ and $t$,
where the expectation is over the random variables $\{w^{i,X}_t\}$.
We will show in Section~\ref{sec:sim} that it is possible to reduce the social cost
by setting the parameter $\sigma$ to a non-trivial value.

\subsection{Agent Population and Policies} 
In response to the history of
signals received prior to time $t$, every agent $i$ takes action
$a_t^i$.  We assume that every agent acts
based only on the signals, without considering the response of other
agents to its own action.  This is a reasonable assumption for three reasons.  
First, it is hard for the agent to obtain more information
than the signal sent by the central agent.  Second, the agents know
that the signals received are truthful.  
Finally, the environment may be rapidly changing, 
making independent decisions difficult. 

Formally, let $S^i_t$ denote the history of signals received by agent
$i$ up to time $t$: $ S^i_t \triangleq \{s^i_1, \ldots, s^i_{t-1}\}$.
Let $\mathcal S_t$ denote the set of possible realisations of signal
histories 
up to time $t$.  
A sequence of mappings from signal history to action, $\mathcal
S_t \to \{A,B\}$ for $t=1,2,\ldots$, is called a policy.  

We let $\Omega$ denote the set
of all possible types of agents. 
We introduce a probability measure
$\mu$ over the set $\Omega$, which describes the distribution of the
population of agents into types.  We assume that each type $\omega
\in \Omega$ is associated with a policy, and every agent of type
$\omega$ follows the policy $\pi^\omega: \mathcal S_t \to \{A,B\}$.
For a population of $N$ agents,
we assume that $\mu$ is such that
  $\mu(O) N$ is a non-negative integer for all $O\subseteq \Omega$.
Consequently, $\mu(\omega) N$ denotes the number of agents
with policy $\pi^{\omega}$.  The following assumption simplifies the
analysis.
\begin{assumption}\label{as:pop}
  The true distribution $\mu$ is known to the central agent.
\end{assumption}
It is also reasonable that the central agent is capable of estimating
the distribution of agents $\mu$ using statistical estimation
techniques.  Estimating distribution of agents is possible by Bayesian
or maximum a-posteriori methods.  If the estimation is not accurate,
the central agent can always fall back to sending no signal, which
creates no more congestion than there already is.

\subsubsection{$\pi^k$-policies} In the case of scalar signalling, we assume that
the set of possible types is $\Omega = \{1,2,\ldots\}$.  
Each agent of type $k\in \Omega$ has an actuation delay of $k$ time
steps, i.e., it acts at time $t$
upon the signal $s^i_{t-k}$ received at the time step $t-k$ -- if the
latter is defined.
This effectively models the delay between when an agent makes the
decision which action to take, and when it contributes to congestion.
These delays do not come from communication, but from the
agent-specific delay between making a decision and causing congestion.
For example, a driver may decide to take route $A$ at time $t-k$, but
begin the journey only at time $t$, resulting in the delay between
deciding on a route and reaching a particular congested road segment.
Another example is the computational delay between receiving stock
market information and deciding to invest in a stock.  

More precisely, in response to
the scalar signalling scheme with parameter $\sigma$, i.e., $s^i_t :=
(y^{A,i}_t, y^{B,i}_t)$, each agent $i$ of type $k$
acts according to the following policy $\pi^k$:
\begin{align}
  a_t^i &= \pi^k(S^i_{t}) = \pi^k(s^i_{t-k}) \nonumber\\
  &= 
  \begin{cases}
    \arg\min_{X=A,B}  \quad y^{X,i}_{t-k}, &\quad\mbox{if }t \geq k+1,\\
    A, &\quad\mbox{otherwise},
  \end{cases}\label{eq:d1}
\end{align}
i.e., the agent chooses greedily the action with the smallest cost signal at
time $t-k$.

\subsection{Guarantees}\label{sec:sim}

In this section, we characterise the social cost resulting from
$\sigma$-scalar signalling and a population of agents following
policies $\{\pi^k : k=1,2,\ldots\}$.  Furthermore, we show that the
expected distribution of agents between the actions converges for
every initial condition.

For clarity of exposition, we first consider a homogeneous population.
We analyse the expected value and concentration property of the social
cost, and show the convergence of each agent's action profile.
Our analysis methods can be used to obtain corresponding guarantees
for the heterogeneous population cases.
We also illustrate our main theoretical guarantees with simulations.

\subsubsection{Homogeneous Population} 
In this section, we consider a homogeneous population
of agents with the policy $\pi^1$, i.e., the case $\mu(1) = 1$.  In other words, we assume that every
agent $i$ acts according to $\pi_1$.

First, we focus on the next-step outcome and establish closed-form
expressions for the expectation of the number of agents taking action
$A$.

\begin{lemma}[Conditional Distribution of $n^A_t$]\label{le:1}
  Suppose that $n^A_{t-1}$ takes a fixed value $n\in\{1,\ldots,N\}$
  and that all agents are of policy $\pi^1$. Consider an arbitrary
  time step $t>1$ and arbitrary $n^A_{t-1}$.  Let $p_{\sigma,n}
  \triangleq \Phi_\sigma(c_B(N-n) - c_A(n))$.  Then, we have
  \begin{align*}
    \P(n^A_t = m \mid n^A_{t-1} = n) = \binom{N}{m} p_{\sigma,n}^m
    (1-p_{\sigma,n})^{N-m}.
  \end{align*}
\end{lemma}


\begin{proof}
  Recall that $y^{A,i}_1 := c_A(n^A_t) + w_i$, and $y^{B,i}_1 :=
  c_B(N-n^A_t) + w_i$.  Notice that $a_t^i$ is a random variable. The
  number of agents taking action $A$ at time $t$ is:
  \begin{align*}
    n^A_{t} = & \sum_{i = 1}^{N} 1_{ [a^i_t = A] }.
  \end{align*}
  Clearly, for a fixed $t$, $\{w^A_{i,t} - w^B_{i,t}:i=1,\ldots,N\}$
  is \iid with normal distribution $N(0,\sigma^2)$. Hence, for a fixed
  $t$, $\{ 1_{[a^i_t = A]} : i=1,\ldots,N \}$ are \iid Bernoulli
  random variables, each with parameter $\P(a^i_t = A)$.

  Let $\Phi_\sigma$ denote the tail probability function of the
  distribution $N(0,\sigma^2)$.  Observe that
  \begin{align*}
    \P(a^i_t = A) &= \P( y^{A,i}_t  < y^{B,i}_t )\\
    &= \P(c_A(n^A_{t-1}) + w^A_{i,t} < c_B(N-n^A_{t-1}) + w^B_{i,t})\\
    &= \P(w^A_{i,t} - w^B_{i,t} < c_B(N-n^A_{t-1}) - c_A(n^A_{t-1}) )\\
    &= \Phi_\sigma(c_B(N-n) - c_A(n)) = p_{\sigma,n},
  \end{align*}
  where we used the assumption that $n^A_{t-1} = n$.

  Since $n^A_t$ is a sum of Bernoulli random variables, the claim
  follows by the binomial probability mass function.
\end{proof}

Next, we can derive a closed-form expression for the expectation of
the next-step social cost.

\begin{theorem}[Expected Next-step Social Cost]\label{pro:1}
  Let us take the assumptions as in Lemma~\ref{le:1}.  
For every parameter $\sigma > 0$, and
  allocation $n\in \{1,\ldots,N\}$ to action $A$, we have
  \begin{align*}
    \frac{\E C(n^A_t)}{C(n)} = \sum_{m = -\infty}^{\infty}
    \frac{C(m)}{C(n)} \binom{N}{m} p_{\sigma,n}^m
    (1-p_{\sigma,n})^{N-m},
  \end{align*}
  where the expectation is conditioned on $n^A_{t-1} =
  n$.
\end{theorem}

\begin{proof}
  Observe that
  \begin{align}
    \E C(n^A_t)
    &= \sum_{m = -\infty}^{\infty} C(m) \P(n^A_t = m)\nonumber\\
    &= \sum_{m = -\infty}^{\infty} C(m) \binom{N}{m} p_{\sigma,n^*}^m
    (1-p_{\sigma,n^*})^{N-m}.\label{eq:767}
  \end{align}
  The claim follows.
\end{proof}

\begin{figure}[t]
  \begin{subfigure}[t]{0.5\textwidth}
    \includegraphics[width=0.95\textwidth]{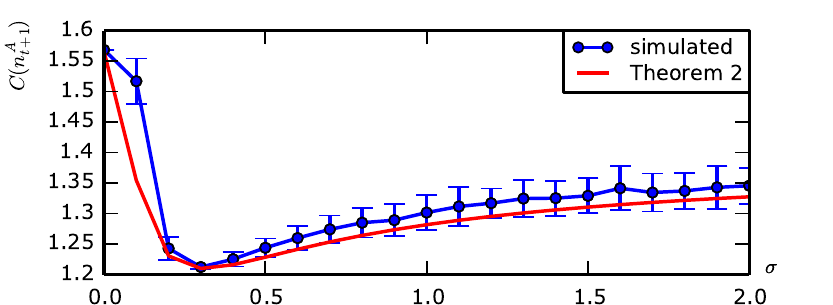}
    \caption{Next-step social cost $C(n^A_t)$ versus parameter
      $\sigma$ with $n^A_{t-1} = n^*$.  Observe that the social cost is minimum at
      $\sigma = 0.3$ and that the variance decreases as we approach
      the minimum.  }
    \label{fig:e}
  \end{subfigure}
  \hskip 3mm
  \begin{subfigure}[t]{0.5\textwidth}
    \includegraphics[width=0.95\textwidth]{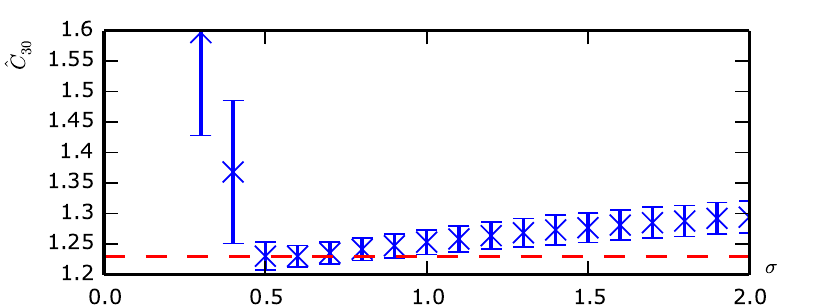}
    \caption{Time-averaged social cost $\hat C_{30}$ versus parameter
      $\sigma$ with $n^A_1 = n^*$.  The dashed horizontal line is at
      the unique optimum.}
    \label{fig:ushaped5longrun}
  \end{subfigure}
  \caption{$\sigma$-scalar signalling, homogeneous population of
    $N=40$, and the cost functions of
    Figure~\ref{fig:ushaped5settings_funs}.
}
\end{figure}


Recall that $n^*$ denotes the socially optimal allocation of agents to
action $A$.
Figure~\ref{fig:e}\footnote{In Figure~\ref{fig:e} and throughout the paper, all
  error-bar plots are averaged over 100 simulations with error bars
  corresponding to one standard deviation.} illustrates
the expected social cost of Theorem~\ref{pro:1}
as a function of the parameter $\sigma$,
along with the simulated mean and standard deviation of $C(n^A_t)$.
In particular, observe that by setting $\sigma = 0.3$, the next-step
social cost $C(n^A_t)$ achieves the
theoretical minimum social
cost $C(n^*)$ of Figure~\ref{fig:e}, provided that the population
distribution is optimal at the current time step ($n_{t-1}^A = n^*$).
Furthermore, Figure~\ref{fig:ushaped5longrun} shows the time-averaged
social cost $\hat C_T$ can achieve values close to the theoretical
minimum social cost by setting the parameter $\sigma$ appropriately.
Hence, $\sigma$-scalar signalling is useful even when we do not satisfy
the condition $n_{t-1}^A = n^*$ for many time indices $t$.

\begin{remark}[Optimal $\sigma$]
  It is clear from Figure~\ref{fig:e} and by substituting $p_{\sigma,n}
  \triangleq \Phi_\sigma(c_B(N-n) - c_A(n))$ into Theorem~\ref{pro:1}
  that we can readily minimise the next-step expected social cost by
  optimising the amount of noise 
  $\sigma$.
\end{remark}

Next, we quantify the concentration of the social cost 
around its expected value, which is already hinted at 
in the error bars of Figure~\ref{fig:e}.

\begin{theorem}[Concentration of Next-step Social Cost]\label{thm:1}
  Suppose that the assumptions of Theorem~\ref{pro:1} and
  that the functions $c_A$ and $c_B$ are Lipschitz with constant $L$.  We have
  \begin{align*}
    \P\left( \abs{ \frac{C(n^A_t)}{C(n^*)} - \frac{\E
          C(n^A_t)}{C(n^*)} } \geq \eps \right) \leq 2 \exp
    \left(-\frac{2 \eps^2 C(n^*) }{2(L+1)}\right),
  \end{align*}
  where the probability is conditioned on $n^A_{t-1} = n^*$.
\end{theorem}

\begin{proof}
  Let
  \begin{align*}
    X_i = 1_{[a^i_t = A]}, \quad\mbox{for all }i=1,\ldots,N,
  \end{align*}
  so that $n^A_t = \sum_i X_i$.  Let $f$ be such that
  \begin{align*}
    &f(X_1,\ldots,X_N)\\
    &= C(n^A_t) / C(n^*)\\
    &= \frac{1}{C(n^*)} \Big[\frac{\sum_i X_i}{N} \cdot
    c_A\left(\frac{\sum_i X_i}{N}\right)\\
    &\quad\quad+ \frac{N-\sum_i X_i}{N} \cdot c_B\left(\frac{N-\sum_i
        X_i}{N}\right) \Big].
  \end{align*}
  Let $x_1,\ldots,x_N,\hat x_i \in \{0,1\}$.  Observe that since $c_A$
  and $c_B$ take values in $[0,1]$ and are $L$-Lipschitz by
  assumption, by simple algebra, we obtain
  \begin{align*}
    &\sup \abs{f(x_1,\ldots,x_N) - f(x_1,\ldots,\hat x_i,\ldots,x_N)}\\
    &\leq \frac{2(1/N + L / N)}{C(n^*)},
  \end{align*}
  for all $i$.  Hence, by McDiarmid's inequality, we obtain the claim.
\end{proof}

Next, 
  we consider the normalised random process $n^A_t / N$
for $t = 0, 1, 2, \ldots$ In particular, we consider the expectation
of that process, i.e., the sequence
\begin{align*}
  x^i_t \triangleq \P(a^i_t = A) = \E \frac{n^A_t}{N}, \quad
  t=1,2,\ldots
\end{align*}
We show that this sequence converges.
In order words, the action profile of every agent converges under
$\sigma$-scalar signalling if every agent follows the $\pi^1$ policy \eqref{eq:d1}.

\begin{theorem}[Convergence]\label{thm:conv}
  Suppose that the assumptions of Theorem~\ref{le:1} hold.  For every
  initial condition $x^i_1 = \P(a^i_1 = A)$, the sequence $x^i_t =
  \P(a^i_t = A)$ converges as $t\to\infty$.
\end{theorem}

\begin{proof}
  Observe that, for all $t > 2$, we have
  \begin{align}
    x^i_t &= \P(a^i_t = A)\\
    &= \P(w^A_{i,t} - w^B_{i,t} < c_B(N-n^A_{t-1}) -
    c_A(n^A_{t-1}) )\\
    &= \sum_m \P(w^A_{i,t} - w^B_{i,t} < c_B(N-m) - c_A(m) )
    \P(n^A_{t-1} = m)\\
&= \sum_m 
\Phi_\sigma(c_B(N-m) - c_A(m)) \P(n^A_{t-1} = m),\label{eq:49}
  \end{align}
  where
  \begin{align}\label{eq:50}
    \P(n^A_{t-1} = m) &= \binom{N}{m} \P(a^i_{t-1})^m
    (1-\P(a^i_{t-1}))^{N-m}\\
    &= \binom{N}{m} (x^i_{t-1})^m (1-x^i_{t-1})^{N-m}.
  \end{align}

  By combining \eqref{eq:49} and \eqref{eq:50}, we obtain a function
  $f:\mathbb R \to \mathbb R$ that describes the iterated function
  process $\{x^i_t : t=1,2,\ldots\}$:
  \begin{align*}
    x^i_t = f(x^i_{t-1}), \quad\mbox{for all }t.
  \end{align*}
  It is easy to verify that there exists an $\ell < 1$ such that for
  all $x,y$,
  \begin{align*}
    \abs{f(x) - f(y)} \leq \ell \abs{x - y}.
  \end{align*}
  Hence, $f$ is a contraction and the claim follows from Banach
  fixed-point theorem.
\end{proof}

\begin{figure}[t]
  \begin{subfigure}[t]{0.5\textwidth}
    \includegraphics[clip=true,width=\textwidth]{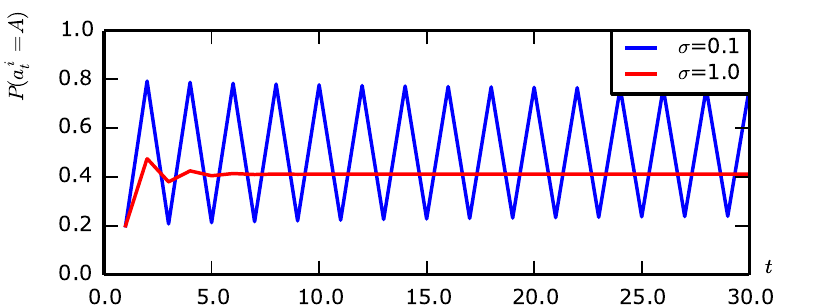}
    \caption{Iterated function process $x^i_t$ for various values of
      $\sigma$.  }
    \label{fig:d}
  \end{subfigure}
  \begin{subfigure}[t]{0.5\textwidth}
    \includegraphics[clip=true,width=\textwidth]{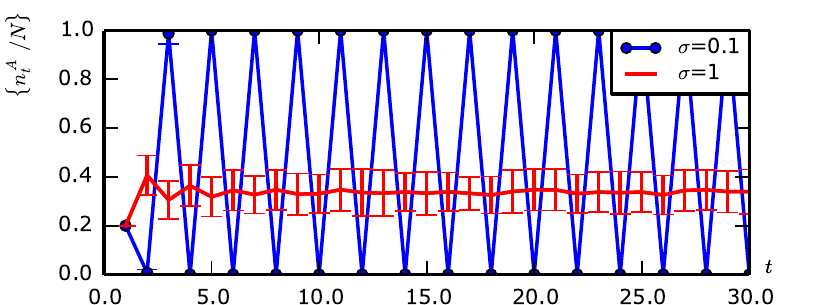}
    \caption{The process $\{n^A_t/N\}$ over time.}
    \label{fig:ushaped5simulations_states}
  \end{subfigure}
  \caption{$\sigma$-scalar signalling, homogeneous population of
    $N=40$, and the cost functions of
    Figure~\ref{fig:ushaped5settings_funs}.
}
  \label{fig:ushaped5simulations}
\end{figure}

To illustrate Theorem~\ref{thm:conv}, Figure~\ref{fig:d} shows the
process $x^i_t$ for various values of $\sigma$, whereas
Figure~\ref{fig:ushaped5simulations_states} shows the related random
process $\{n^A_t/N\}$.
Observe that for the parameter value $\sigma=0.6$, the
process $\{n^A_t/N\}$ converges to values close to the optimal
population profile $n^*/N = 0.2$
(cf.~Figure~\ref{fig:ushaped5settings_costs}), whereas for other
parameter values, we see flapping in the population profiles and
values away from the optimum.

As a corollary of Theorem~\ref{thm:conv}, we obtain the following.

\begin{corollary}
  The random processes $n_t^A/N$ and $C(n_t^A)$ converge in distribution as $t\to\infty$.
\end{corollary}

\begin{remark}[Limit of $C(n_t^A)$]
  It follows from Theorem~\ref{thm:conv} that the limit distribution
  of the random process $C(n_t^A)$ depends on the parameter $\sigma$,
  which can be optimised.
\end{remark}

\begin{figure}[t]
\begin{subfigure}[t]{0.5\textwidth}
  \includegraphics[clip=true,width=0.99\textwidth]{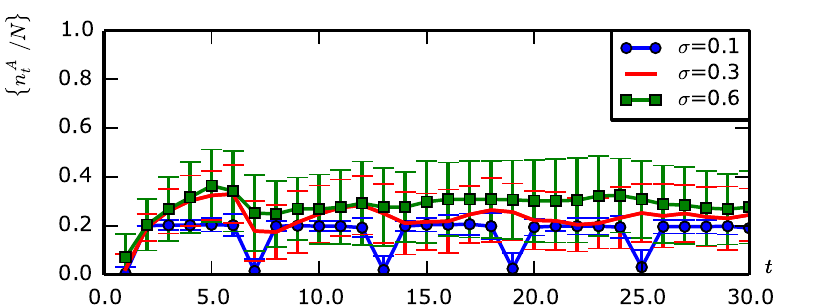}
  \caption{The processes
    $\{n^A_t/N\}$ over time for a uniform population distribution
    $\mu$ over delays of $1, 2, \ldots, 5$. 
}
  \label{fig:ushaped5simulations_delayed20}
\end{subfigure}
\begin{subfigure}[t]{0.5\textwidth}
  \includegraphics[clip=true,width=0.99\textwidth]{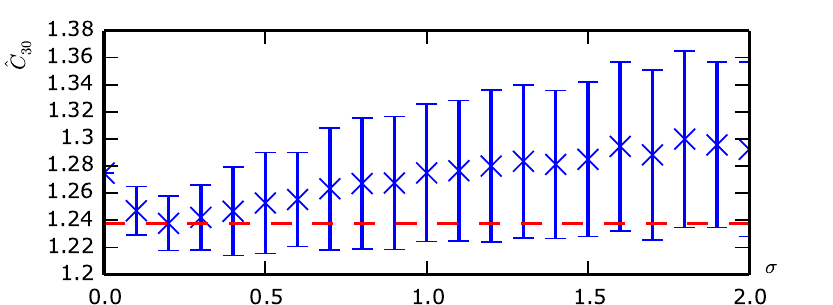}
                \caption{
The time-averaged social cost $\hat C_{30}$ versus parameter $\sigma$
for 
uniform population distribution $\mu$ over delays of $1, 2, \ldots, 20$. 
}
\label{fig:ushaped5simulations_delayed_longrun_quick}
\end{subfigure}
\caption{
$\sigma$-scalar signalling on the cost functions of
Figure~\ref{fig:ushaped5settings_funs} with $N=40$ agents.}
\end{figure}

\subsubsection{Heterogeneous Population}
In this section, we consider a distribution $\mu = (\mu_1,\ldots,\mu_K)$ of agents over the
policies $\pi^1,\ldots,\pi^K$, i.e., with $\mu_j N$ agents with delays
$j$ for $j=1,\ldots,K$.  We first derive the following analogue of
Lemma~\ref{le:1}.  

\begin{lemma}[Conditional Distribution of $n^A_t$]\label{le:2}
  Suppose that Assumption~\ref{as:pop} holds. Consider an arbitrary
  time step $t>K$ and suppose that $n^A_{t-k} = n_0$ for all
  $k=1,\ldots,K$.  Let $p_\sigma \triangleq \Phi_\sigma(c_B(N-n_0) -
  c_A(n_0))$.  Then, we have
  \begin{align*}
    &\P(n^A_t = m \mid n^A_{t-k} = n_0 \mbox{ for }k=1,\ldots,K)\\
    &= \sum_{\abs{\alpha} = m} \binom{m}{\alpha} \prod_{j=1}^K
    \binom{\mu_j N}{\alpha_j} p_\sigma^{\alpha_j} (1-p_\sigma)^{\mu_j
      N-\alpha_j},
  \end{align*}
  where $\alpha = (\alpha_1,\ldots,\alpha_K)$ denotes a multi-index.
\end{lemma}

\begin{proof}
  Let $n^A_{j,t}$ denote the number of agent of type $j$ who choose
  action $A$ at time $t$.  By the same argument as Lemma~\ref{le:1},
  we have
  \begin{align*}
    \P(n^A_{j,t} = m) = \binom{\mu_j N}{m} p_\sigma^m
    (1-p_\sigma)^{\mu_j N-m}, \quad j=1,\ldots,K.
  \end{align*}
  By definition, we have
  \begin{align*}
    \P(n^A_t = m) &= \P\left(\sum_{j=1}^K n^A_{j,t} = m \right)\\
    &= \sum_{\abs{\alpha} = m} \binom{m}{\alpha} \P(n^A_{1,t} =
    \alpha_1) \ldots \P(n^A_{K,t} = \alpha_K).
  \end{align*}
  The claim follows by algebra.
\end{proof}

Using Lemma~\ref{le:2}, we can readily derive analogues of 
  the expected next-step social cost (Theorem~\ref{pro:1}), the
concentration of
  social cost (Theorem~\ref{thm:1}), and the convergence of action profiles (Theorem~\ref{thm:conv}) guarantees for heterogeneous 
  populations.
Figures~\ref{fig:ushaped5simulations_delayed20} illustrates the
convergence of the population profile $n_t^A/N$ with a heterogeneous
population under $\sigma$-scalar signalling.
Observe that, as in
    Figure~\ref{fig:ushaped5simulations_states}, a cyclic
    flapping behaviour is visible at low values of $\delta$, e.g., $\delta = 0.1$.

An empirical study of the time-averaged social cost $\hat C_t$ shows a
dependence on the parameter $\sigma$ similar to the next-step social
cost $C(n_t^A)$, which can be seen by comparing
Figure~\ref{fig:ushaped5simulations_delayed_longrun_quick} to 
Figure~\ref{fig:e}.

\section{Interval Signalling}
\label{sec:mod3}

In this section, we model agents with various levels of risk aversion,
and present a broadcast signalling scheme that desynchronises
them. Each signal is composed of pairs of endpoints, hence, the name
interval signalling.
The overall closed-loop system composed of the central agent and the
agents is shown in Figure~\ref{fig:diag2}.

\begin{figure}
  \centering
\includegraphics[width=0.99\textwidth,clip=true,trim=8.5cm 8cm 5.5cm 4.5cm]{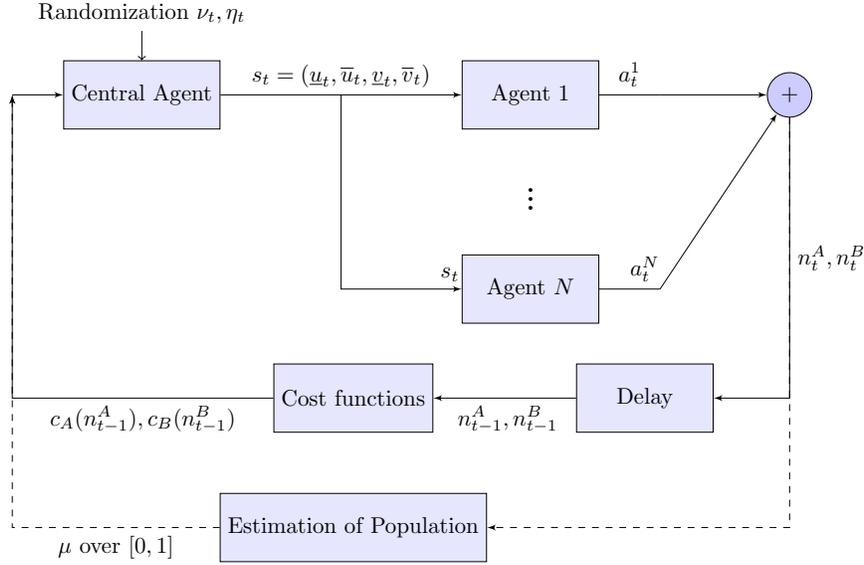}
  \caption{Block diagram for interval signalling.}
  \label{fig:diag2}
\end{figure}

\subsection{$(\delta,\gamma)$-Interval signalling} 
In this section, we present a
signalling scheme that broadcast the same signal to all agents.  Let
$\delta$ and $\gamma$ denote non-negative constants.  In the
$(\delta,\gamma)$-interval signalling scheme, a central agent
broadcasts to all agents the same signal $s^i_t = s_t$ for all
$i$. 
This is in constrast to $\sigma$-scalar signalling, where the central agent sends a 
distinct signal to every agent $i$.
For every time step $t$, the signal $s_t = (\underline{u}^A_t, \overline{u}^A_t,
\underline{u}^B_t, \overline{u}^B_t) \in \mathbb R^4$ is defined as follows\footnote{This corresponds
  to the case $d=4$ in \eqref{eq:81}.}:
\begin{align}
  \underline{u}^A_t \triangleq c_A(n^A_{t-1}) + \nu_t - \delta/2,\quad
  &
  \overline{u}^A_t \triangleq c_A(n^A_{t-1}) + \nu_t + \delta/2,\\
  \underline{u}^B_t \triangleq c_B(n^B_{t-1}) + \eta_t -
  \gamma/2,\quad & \overline{u}^B_t \triangleq c_B(n^B_{t-1}) + \eta_t
  + \gamma/2.
\end{align}
where $\nu_t$ and $\eta_t$ are \iid uniform random variables with
supports:
\begin{align*}
  \mbox{Supp}(\nu_t) &= [-\delta/2,\delta/2],\quad \mbox{Supp}(\eta_t)
  = [-\gamma/2,\gamma/2].
\end{align*}
Figure~\ref{fig:is} illustrates the interval $[\underline{u}^A_t ,
\overline{u}^A_t]$.

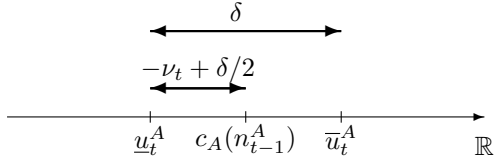
\begin{figure}
  \centering \setlength{\unitlength}{.01in}
  \begin{picture}(200,90) \put(0,25){\vector(1,0){250}}
    \put(75,22){\line(0,1){6}}
    \put(125,22){\line(0,1){6}}
    \put(175,22){\line(0,1){6}}
    \thicklines
    \put(250,5){\makebox(0,0)[b]{$\mathbb R$}}
    \put(75,5){\makebox(0,0)[b]{$\underline{u}^A_t$}}
    \put(125,5){\makebox(0,0)[b]{$c_A(n^A_{t-1})$}}
    \put(175,5){\makebox(0,0)[b]{$\overline{u}^A_t$}}
    \put(75,40){\vector(1,0){50}}
    \put(125,40){\vector(-1,0){50}}
    \put(100,50){\makebox(0,0){$-\nu_t+\delta/2$}}
    \put(75,70){\vector(1,0){100}}
    \put(175,70){\vector(-1,0){100}}
    \put(120,80){\makebox(0,0){$\delta$}}
  \end{picture}  
  \caption{An illustration of interval signalling.}
  \label{fig:is}
\end{figure}

By construction, it is clear that $(\delta,\gamma)$-interval signalling
scheme has the following
property, which we call \emph{interval-truthfulness}:
\begin{align*}
  c_A(n^A_{t-1}) &\in [\underline{u}^{A,i}_t,
  \overline{u}^{A,i}_t] \mbox{ and } c_B(n^B_{t-1}) \in
  [\underline{u}^{B,i}_t, \overline{u}^{B,i}_t]
\end{align*}
with probability $1$ for all $i$ and $t$.
We will show in Section~\ref{sec:int} that it is also possible to pick the 
parameters $(\delta,\gamma)$ to reduce the social cost.

\subsubsection{$\pi^\lambda$-policies} 

In the case of interval signalling, we
consider a set of agent types $\Omega = [0,1]$, where the type
captures a level of risk aversion.
Recall that every agent receives the same interval signal $s_t :=
(\underline{u}_t, \overline{u}_t, \underline{v}_t, \overline{v}_t)$.
In response, we assume that each agent of type $\lambda$ follows
the policy $\pi^\lambda$:
\begin{align}\label{eq:p2}
  a_t^i = \pi^\lambda(S^i_{t}) = \pi^\lambda(s_t) := \arg\min_{X \in
    \{A,B\}} \lambda \underline{u}^X_t + (1-\lambda) \overline{u}^X_t.
\end{align}
These policies naturally model a notion of risk aversion.  Observe that for $\lambda = 0$, policy
$\pi^0$ models a risk-averse agent, who makes decisions based solely
on upper endpoints $\overline{u}^A_t$ and $\overline{u}^B_t$ of the
respective intervals.  Similarly, $\pi^1$ and
$\pi^{1/2}$ model risk-seeking and risk-neutral agents.  Although the
extremes may be rare, it seems plausible that people combine the
optimistic and pessimistic views in this fashion.  The following
example illustrates these policies.


\begin{example}
  Suppose that at time $t=1$, we have $c_A(n^A_1) = 1$ and $c_B(n^B_1)
  = 1$. For $\delta = 1, \gamma = 0.5$, the interval signal $s_{2} :=
  (\underline{u}_2, \overline{u}_2, \underline{v}_2, \overline{v}_2)$
  at time $2$, could take the realisation of $(0.5, 1.5, 0.8, 1.3)$.
  The risk-seeking agent with $\lambda = 1$ would pick $A$, whereas
  the risk-averse agent with $\lambda = 0$ would pick $B$.
\end{example}

\subsection{Guarantees}\label{sec:int}

In this section, we consider $(\delta,\gamma)$-interval signalling and
$\pi^\lambda$-policies.  For signalling scheme that
  broadcasts the same signal to all agents, the outcome corresponding
  to homogeneous populations is trivial since all agents would pick the
  same action.  Hence, we consider
  directly a heterogenous population over a set of types $[0,1]$ and
  a population distribution $\mu$.

First, we derive a lemma for the conditional probability of
taking action $A$.
The main result of this section is an expression to
evaluate the expected next-step social cost for a $\mu$-heterogeneous
population under $(\delta,\gamma)$-interval signalling.

\begin{lemma}[Conditional Probability of Action $A$]\label{le:3}
  Let $F_{\delta,\gamma,n}(\lambda) \triangleq \P(\pi^\lambda(s_t) = A
  \mid n^A_{t-1} = n)$ denote, for every $\lambda\in \Omega$, the
  probability of an agent with policy $\pi^\lambda$ choosing action
  $A$ conditioned on the realisation of the random variable
  $n^A_{t-1}$. We have
  \begin{align*}
    F_{\delta,\gamma,n}(\lambda)
    &= \P(\nu_t + \eta_t >c_A(n) + \delta(1/2-\lambda)\\
    &\quad- c_B(N-n) - \gamma(1/2-\lambda)),
  \end{align*}
  where $\nu_t$ and $\eta_t$ are independent uniform random variables
  with supports $[-\delta/2,\delta/2]$ and $[-\gamma/2,\gamma/2]$.
\end{lemma}

\begin{proof}
  For every $\lambda\in \Omega$, the probability of an agent with
  policy $\pi^\lambda$ choosing action $A$ is
  \begin{align*}
    &\P(\pi^\lambda(s_t) = A)\\
    &= \P(\lambda \underline{u}^A_t + (1-\lambda) \overline{u}^A_t <
    \lambda \underline{u}^B_t +
    (1-\lambda) \overline{u}^B_t)\\
    &= \P(c_A(n^A_{t-1}) + \delta(1/2-\lambda) - c_B(N-n^A_{t-1}) -
    \gamma(1/2-\lambda)\\
    & \quad < -\nu_t + \eta_t),
  \end{align*}
  where the first equality follows by definition of the policy
  $\pi^\lambda$, and the second equality follows by the definition of
  $(\delta,\gamma)$-interval signalling and simple algebra. The claim
  follows from the fact that $-\nu_t$ has the same distribution as
  $\nu_t$ by definition.
\end{proof}

\begin{theorem}[Expected Next-Step Social Cost]\label{pro:3}
  Suppose that the assumptions of Lemma~\ref{le:3} hold with $n =
  n^*$.  Let $p_{\delta,\gamma,n^*} \triangleq \int_{\lambda \in
    [0,1]} F_{\delta,\gamma,n^*}(\lambda) \mu(\ud \lambda)$.  We have
  \begin{align*}
    \frac{\E C(n^A_t)}{C(n^*)} = \sum_{m = -\infty}^{\infty}
    \frac{C(m)}{C(n^*)} \binom{N}{m} p_{\delta,\gamma,n^*}^m
    (1-p_{\delta,\gamma,n^*})^{N-m}.
  \end{align*}
\end{theorem}

\begin{proof}
  By the definition $F_{\delta,\gamma,n^*}(\lambda) \triangleq
  \P(\pi^\lambda(s_t) = A \mid n^A_{t-1} = n^*)$, we have
  \begin{align*}
    \P(n^A_t = m \mid n^A_{t-1} = n^*) = \binom{N}{m}
    p_{\delta,\gamma,n^*}^m (1-p_{\delta,\gamma,n^*})^{N-m}.
  \end{align*}
  The claim follows by the same argument as the proof of
  Theorem~\ref{pro:1} -- cf. \eqref{eq:767}.
\end{proof}

Figure~\ref{fig:int_ushaped5heatmaps8} presents in heat-map form the dependence of
the next-step social cost $C(n^A_{t+1})$ on the parameters $\delta$
and $\gamma$ of the
interval signalling scheme.
Clearly, the optimal $(\delta,\gamma)$-values are a non-trivial region
bounded away from $0$ and $\infty$.
The simulated next-step social cost $C(n^A_{t+1})$ on
Figure~\ref{fig:int_ushaped5heatmaps8_0} coincides with the
expected value of $C(n^A_{t+1})$ computed using Theorem~\ref{pro:3} on
Figure~\ref{fig:int_ushaped5heatmaps8_2}.

\begin{remark}[Concentration of $C(n^A_t)/C(n^*)$]
 In the case of $(\delta,\gamma)$-interval signalling, 
the social cost
  ratio $C(n^A_t)/C(n^*)$ is concentrated around the mean in a similar
  fashion to Theorem~\ref{thm:1}.
\end{remark}

\subsection{Long-Run Behaviour}

In this section, we study empirically the long-run behaviour of the state
process $n_t^A/N$ under $(\delta,\gamma)$-interval signalling.
Figure~\ref{fig:int_ushaped5simulations} shows the simulated the
time evolution of $n_t^A/N$.
As in
    Figure~\ref{fig:ushaped5simulations_states}, 
    a severe flapping effect is exhibited at a low value of $\delta$, e.g., $\delta = 0.5$. For higher values of $\delta$,
   there are hints of convergence.

Figure~\ref{fig:int_ushaped5longrun_30_lambda-random_gamma0.2}
shows the dependence of the time-averaged social cost $\hat C_T$ on
one parameter $\delta$, with the other parameter fixed at $\gamma=0.2$. 
Observe that the dependence is similar to that exhibited in the case
of $\sigma$-scalar signalling (cf.~\ref{fig:e}); moreover, at the
optimal value of the parameter $\delta$, the social cost -- although
random -- is close to the theoretical minimum (cf.~\ref{fig:ushaped5settings_costs}).



\begin{figure}[t]
  \begin{subfigure}[t]{0.5\textwidth}
    \includegraphics[clip=true,width=0.99\textwidth]{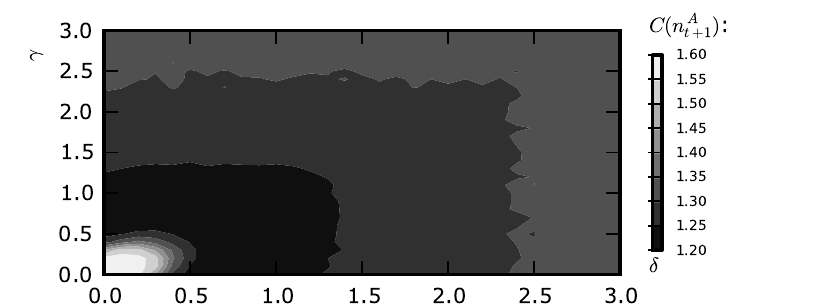}
    \caption{Simulated next-step social cost $C(n^A_{t+1})$ versus
      $\delta$ and $\gamma$.}
    \label{fig:int_ushaped5heatmaps8_0}
  \end{subfigure}
  \begin{subfigure}[t]{0.5\textwidth}
    \includegraphics[clip=true,width=0.99\textwidth]{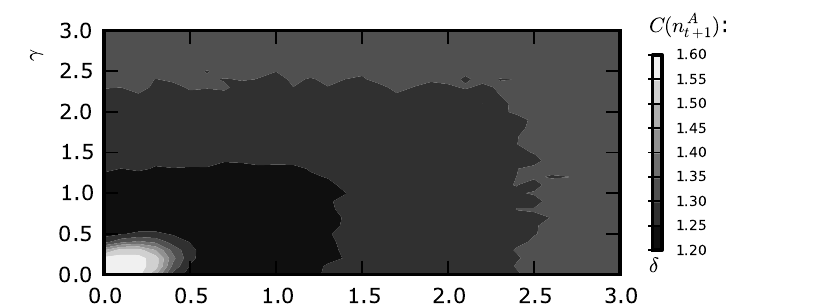}
    \caption{Expected $C(n^A_{t+1})$ computed using Theorem~\ref{pro:3} as a
      function of $\delta$ and $\gamma$}
    \label{fig:int_ushaped5heatmaps8_2}
  \end{subfigure}
  \caption{$(\delta, \gamma)$-interval signalling, cost functions of
    Figure~\ref{fig:ushaped5settings_funs}, and heterogeneous
    population with $N=40$ and $\mu$ uniform over $\{1/N,\ldots,1\}$.
  }
  \label{fig:int_ushaped5heatmaps8}
\end{figure}

\begin{figure}[t]
 \begin{subfigure}[t]{0.5\textwidth}
\includegraphics[clip=true,width=0.99\textwidth]{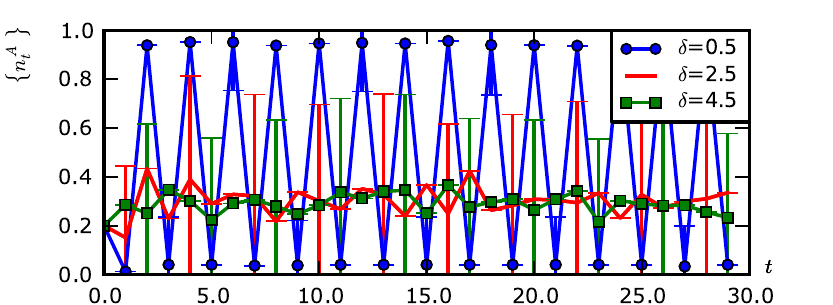}
\caption{The processes $\{n^A_t/N\}$ over time for $N = 40$ agents,
  $\delta = 0.1$, and $\gamma = 0.2$.}
\label{fig:int_ushaped5simulations}
\end{subfigure}
\begin{subfigure}[t]{0.5\textwidth}
  \includegraphics[clip=true,width=0.99\textwidth]{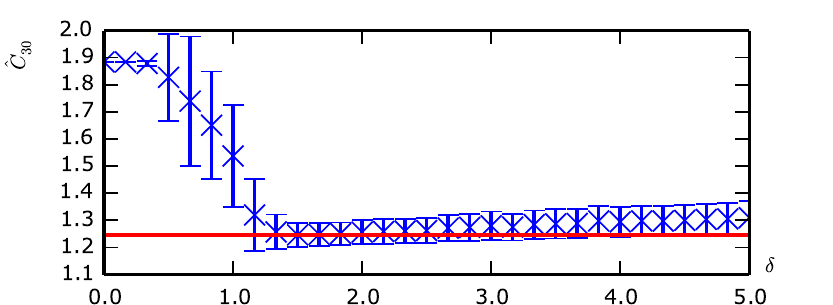}
\caption{Time-averaged social cost $\hat C_{30}$
    versus $\delta$ for a fixed $\gamma = 0.2$,
uniform distribution $\mu$
    over {$[0,1]$}, and $N=40$.  
  The dashed horizontal line is
    at the optimal $\delta$.}
  \label{fig:int_ushaped5longrun_30_lambda-random_gamma0.2}
\end{subfigure}
\caption{$(\delta, \gamma)$-interval signalling on the cost functions of
    Figure~\ref{fig:ushaped5settings_funs}}
\end{figure}


\section{Conclusion}\label{sec:con}

Our analysis and simulations give quantitative guidance on how to
design signalling schemes in order to reduce the social cost.  In
particular, if we can estimate the true population mixture, e.g., using
statistical estimation techniques, then we can optimise the signalling
scheme to minimise the social cost. Although our analysis focuses on
  a congestion problem with two resources, the results can be extended
  to an arbitrary number of resources, e.g., by replacing the binomial
  distributions by a multinomial distribution in the analysis.


In contrast to signals that report past outcomes as deterministic
numbers, which correspond to singular probability distributions, we
propose random signals with non-singular distributions.  When agents
mistake deterministic signals for precise predictions of the future,
their actions lead to the cyclic flapping behaviour.  We show that agents
who receive random signals with non-singular distributions exhibit
cyclic behaviour to a lesser degree, which can improve the social
cost. 

Regarding future work, we pose several new questions that can be
studied by extending our model.  What happens when population size and
composition change over time? This can be modeled by a Markovian
sequence of random variables corresponding to agent types. How do the optimal parameter values
for $\sigma,\delta,\gamma$ depend on those changes? This can
be done by employing state-space control techniques to obtain these optimal values.
Would our results hold if every agent had a distinct cost function, so as to model the priority given to some
agents, e.g., when evacuating the young and elderly?
This requires a new notion of consistency or truthfulness that is
specialised to individual agents, as well as a new analysis.

\section*{Acknowledgement}
This work was
   supported in part by the EU FP7 project INSIGHT under grant
   318225, and in part by Science Foundation Ireland grant
   11/PI/1177.


\bibliographystyle{abbrv}
\bibliography{traffic,traffic-old}

\end{document}